\documentclass[12pt]{amsart}
\usepackage[english]{babel}
\usepackage[utf8]{inputenc}
\usepackage{lipsum}
\usepackage{mathrsfs}
\usepackage{stix}
\usepackage{fullpage}
\usepackage{mathtools}
\usepackage{amsmath}
\usepackage{tikz-cd}
\usepackage[colorlinks]{hyperref}
\usepackage{fullpage}
\usepackage{mathtools}
\usepackage{amsmath}
\usepackage{tikz-cd}
\numberwithin{equation}{section}
\theoremstyle{plain}
\newtheorem{theorem}{Theorem}[section]

\newtheorem{cor}[theorem]{Corollary}

\newtheorem{rem}[theorem]{Remark}

\def\X{\mathbb{Y}}
\allowdisplaybreaks
\begin{document}
\title[Hardy inequalities on metric measure spaces, IV: The case $p=1$]{Hardy inequalities on metric measure spaces, IV: The case $p=1$ }

\author[Michael Ruzhansky]{Michael Ruzhansky}
\address{
	Michael Ruzhansky:
	\endgraf
	Department of Mathematics: Analysis, Logic and Discrete Mathematics
	\endgraf
	Ghent University, Belgium
	\endgraf
	and
	\endgraf
	School of Mathematical Sciences
	\endgraf
	Queen Mary University of London
	\endgraf
	United Kingdom
	\endgraf
	{\it E-mail-} {\rm michael.ruzhansky@ugent.be}
}

\author[A. Shriwastawa]{Anjali Shriwastawa}
\address{
 Anjali Shriwastawa:
 \endgraf
  DST-Centre for Interdisciplinary Mathematical Sciences  \endgraf
  Banaras Hindu University, Varanasi-221005, India
  \endgraf
  {\it E-mail-} {\rm anjalisrivastava7077@gmail.com}}
  
  \author[B.Tiwari]{Bankteshwar Tiwari}
\address{
 Bankteshwar Tiwari :
 \endgraf
  DST-Centre for Interdisciplinary Mathematical Sciences  \endgraf
  Banaras Hindu University, Varanasi-221005, India
  \endgraf
  {\it E-mail-} {\rm banktesht@gmail.com}}



 \begin{abstract} In this paper, we investigate the two-weight Hardy inequalities on metric measure space  possessing polar decompositions for the case $p=1$ and $1 \leq q <\infty.$ This result complements the Hardy inequalities obtained in  \cite{RV} in the case $1< p\le q<\infty.$ The case $p=1$ requires a different argument and does not follow as the limit of known inequalities for $p>1.$ As a byproduct, we also obtain the best constant in the established inequality.   We give examples obtaining new weighted Hardy inequalities on homogeneous Lie groups, on hyperbolic spaces  and on Cartan-Hadamard manifolds for the case $p=1$ and $1\le q<\infty.$ 
 
\end{abstract}
\maketitle
\allowdisplaybreaks
\section{Introduction} \label{Intro}
The main purpose of this article is to study the Hardy type inequalities on metric measure spaces 
for the case $p=1$ and $1\leq q<\infty$. We begin with the brief overview of  Hardy inequalities
in the literature. 
 \\ Let us recall  the  following celebrated Hardy inequality \cite{H2} which plays an important role in many areas such as analysis, probability and partial differential equations:
 \begin{align}\label{eq1.1}
     \int_a^\infty\left(\frac{\int_b^x f(t)\,\text{d}t}{x}\right)^p\text{d}x\le\left(\frac{p}{p-1}\right)^p\int_b^\infty f^p(x)\,\text{d}x,
 \end{align}
where $p>1,\,b>0,$ and $f\ge0$ is a non-negative function. Several forms of the Hardy inequalities \eqref{eq1.1} have been known.The original discrete version of this inequality goes back in \cite{H1}.  Since a lot of work has been done on Hardy inequalities in different forms and different settings, it is difficult to give a complete overview of the literature, let us refer to books and surveys by Opic and Kufner \cite{OK}, Davies \cite{D}, Kufner, Persson and Samko \cite{KPS}. We also refer \cite{BEL,EE,FS,GKPW,GKP,H1,H2,M,PS,RST,RS,RV1} and references therein for a historic overview as well as some recent developments on this subject.

  The first author and Verma \cite{RV} obtained several characterizations of weights for two-weight Hardy inequalities to hold on general metric measure spaces possessing polar decompositions for the range $1<p \leq q<\infty$ (see, \cite{RV1} for the case $0<q<p$ and $1<p<\infty$).  As a consequence, one deduced new weighted Hardy inequalities on $\mathbb{R}^n$, on homogeneous groups, on hyperbolic spaces and on Cartan–Hadamard manifolds. The first author with Kassymov and Suragan  established a reverse version of the integral Hardy inequality on metric
measure spaces with two negative exponents for the range $q \le p < 0$ (see \cite{KRS}).  Recently, the authors of this paper established sharp weighted integral Hardy inequality and conjugate integral Hardy inequality on homogeneous Lie groups  with any quasi-norm for the range $1<p\leq q<\infty$ (see \cite{RST}). We also calculated the precise value of sharp constants in respective inequalities, improving the result of \cite{RV} in the case of homogeneous groups. Let us recall the following Hardy inequality on metric measure spaces with a polar decomposition  from \cite{RV}. A metric space $(\mathbb X,d)$ with a Borel measure $dx$ is said to have  a {\em polar decomposition} at $a\in{\mathbb X}$ if  there is a locally integrable function $\lambda \in L^1_{loc}$  such that for all $f\in L^1(\mathbb X)$ we have
   \begin{equation}\label{EQ:polarintro}
   \int_{\mathbb X}f(x)dx= \int_0^{\infty}\int_{\Sigma_r} f(r,\omega) \lambda(r,\omega) d\omega_{r} dr,
   \end{equation}
    for the set $\Sigma_r=\{x\in\mathbb{X}:d(x,a)=r\}\subset \mathbb X$ with a measure on it denoted by $d\omega_r$, and $(r,\omega)\rightarrow a $ as $r\rightarrow0.$ We denote $|x|_a:=d(x,a)$.
  \begin{theorem}[\cite{RV}] \label{IntHar1}
Let $1<p\le q <\infty$ and let $s>0$. Let $\mathbb{X} $ be a metric measure space with a polar decomposition at $a$. 
Let $u,v> 0$ be measurable functions  positive a.e in $\mathbb{X}$  such that $u\in L^1(\mathbb X\backslash \{a\})$ and $v^{1-p'}\in L^1_{loc}(\mathbb X)$. Denote
\begin{align}
U(x):= { \int_{\mathbb X\backslash{\mathbb{B}(a,|x|_a )}} u(y) dy} \quad\text{and}\quad V(x):= \int_{\mathbb{B}(a,|x|_a  )}v^{1-p'}(y)dy . \nonumber
\end{align} 

Then the inequality
\begin{equation}\label{EQ:Hardy1}
\bigg(\int_\mathbb X\bigg(\int_{\mathbb{B}(a,\vert x \vert_a)}\vert f(y) \vert dy\bigg)^q u(x)dx\bigg)^\frac{1}{q}\le C\bigg\{\int_{\mathbb X} {\vert f(x) \vert}^pv(x)dx\bigg\}^{\frac1p}
\end{equation}
holds for all measurable functions $f:\mathbb{X}\to{\mathbb C}$ if and only if  the following condition holds:

\begin{equation}
    \sup\limits_{x\not=a} \bigg(U^\frac{1}{q}(x) V^\frac{1}{p'}(x)\bigg)<\infty.
\end{equation}

\end{theorem}

In this paper we prove  two weight Hardy inequalities  on general metric measure spaces with a polar decomposition for the case $p=1$ and $1\le q<\infty.$ For the Euclidean space, this was established in \cite{NP}. It is worth noting that we do not assume any doubling condition on the measure.  This inequality will serve as the complementary Hardy inequality of \eqref{IntHar1} for the limiting case. We state the main result of this paper as follows:
\begin{theorem}\label{Th1intr}
Let $1\leq q<\infty$ and let $\mathbb{X}$ be a metric measure space with a polar decomposition satisfying \eqref{EQ:polarintro} at $a$. Let $u,v>0$ be measurable functions positive a.e. in $\mathbb{X}$ such that $u\in L^1(\mathbb{X}\backslash \{a\})$ 
and $v\in L^1_{loc}(\mathbb{X}).$ Then the  inequality
\begin{align}
\left[\int_\mathbb{X}\left(\int_{\mathbb{B}(a,|x|_a)}|f(z)| dz\right)^q u(x)\,dx\right]^\frac{1}{q}\le C\left\{\int_\mathbb{X}|f(x)|\,v(x)\,dx\right\} 
\end{align}
holds for all measurable functions $f:\mathbb{X}\rightarrow\mathbb{C}$ if and only if the following condition holds:
\begin{align}
    A:=\sup_{R>0}\left(\int_{\mathbb{X}\backslash\mathbb{B}(a,R)}u(y)\,dy\right)^\frac{1}{q}\left(\sup_{y\in\mathbb{B}(a,R)} v^{-1}(y)\right)<\infty.
\end{align}
\end{theorem}

Theorem \ref{Th1intr} enables us to obtain interesting and new weighted integral Hardy inequality on homogeneous groups, on hyperbolic spaces and on Cartan-Hadamard manifolds for the case $p=1$ and $1\leq q <\infty$ (see Corollary \ref{cor4.1} and Corollary \ref{cor4.2}).

Apart from Section \ref{Intro}, this manuscript is divided in two sections. In the next section, we will recall the basics of metric measure space and homogeneous Lie groups with some other useful concepts which we will use in our main results. The last section is devoted to presenting important proofs of the main results of this paper.\\  
Throughout this paper, the symbol $A\asymp B$ means that $\exists\,C_{1},C_{2}>0$ such that $C_{1}A\leq B\leq C_{2}A$.

\section{Preliminaries } \label{preli}

In this section, we present a brief overview on the basics of metric measure spaces and homogeneous Lie groups. For more detail on metric measure space  and on homogeneous Lie group as well  of several functional inequalities on metric measure space and on homogeneous Lie groups, we refer to  monographs \cite{FR,FS,RY,RS,RV,RV1,RST,RSY,RS1} and references therein.\\
Let $(\mathbb{X}, d)$ be a metric space with a Borel measure d$x$ allowing
for the following polar decomposition at some fixed point $a \in \mathbb{X}$ and let us  assume that there is a locally
integrable function $\lambda \in L^
1_{loc}$ such that for all $f \in L^1
(\mathbb{X})$ we have

\begin{align}\label{polarform}
 \int_\mathbb{X} f(x) dx=\int_0^\infty\int_{\sum_r} f(r,\omega)\,\lambda(r,\omega) \,\text{d}\omega_r \,\text{d}r
\end{align}
for the sets $\sum_r=\{x\in \mathbb{X}; d(x,a)=r\}$ with a measure on it denoted by $d\omega =d\omega_r.$ The condition \eqref{polarform} is general  and we allow the function $\lambda$ to depend on the whole variable $x=(r,\omega).$\\
Let us briefly provide the motivation behind the condition \eqref{polarform} above. We first note that, when a metric measure space has a differential structure, one can easily deduce  the
function $\lambda(r, \omega)$  as the Jacobian of the polar change of coordinates or by polar decomposition formula on the corresponding differential metric measure space. The situation we are dealing with is very general and we do not assume any differential structure on the metric measure space $\mathbb{X}.$ So it seems natural to assume the decomposition formula \eqref{polarform} on $\mathbb{X}.$ It is worth noting that we do not impose  any doubling condition of the metric measure space $\mathbb{X}.$ Indeed, we provide several examples of $\mathbb{X}$ for which
the condition \eqref{polarform} is satisfied with different expressions for $\lambda(r,\omega)$ and some of them are of non-doubling volume growth such as hyperbolic spaces:
\begin{itemize}
    \item[(i)] {\bf Euclidean Spaces $\mathbb{R}^n$: }We have
    \begin{align*}
    \lambda(r,\omega)=r^{n-1}.
\end{align*}
    \item[(ii)]{\bf Homogeneous Lie groups:}
    If $Q$ is the homogeneous dimension of the homogeneous Lie group $\mathbb{G},$ then
    \begin{align*}
    \lambda(r,\omega)=r^{Q-1}.
\end{align*}
For more detail on homogeneous Lie groups, we refer to \cite{RS,FS}.
\item[(iii)]{\bf Hyperbolic Spaces:}
The hyperbolic space of dimension $n$  denoted by $\mathbb{H}^n,$ is the simply connected, $n$-dimensional Riemannian manifold of constant sectional curvature equal to $-1.$ It satisfies the stronger property of being a symmetric space.
In the case of hyperbolic spaces, we have
\begin{align*}
    \lambda(r,\omega)=(\sinh{r})^{n-1}.
\end{align*}
\item[(iv)]{\bf Cartan-Hadamard manifolds:}
Let $(M,g)$ be a Riemannian manifold. A Riemannian manifold is said to be a Cartan-Hadamard manifold if it is complete, simply connected and has non-positive sectional curvature. If $\mathbb{K}_M$ be the sectional curvature of the Riemannian manifold, then $\mathbb{K}_M\le0$  along each plane section at each point of $M.$ The exponential map $\exp_a:T_aM\rightarrow M$ is a diffeomorphism (Helgason \cite{H}). Let $J(\rho,\omega)$ be the density function on manifold $M,$ then we have the following polar decomposition:
\begin{align*}
    \int_Mf(x)\,dx=\int_0^\infty\int_{\mathbb{S}^{n-1}}
    f(\exp_a(\rho\omega))J(\rho,\omega)\rho^{n-1} \text{d}\rho\, \text{d}\omega,
\end{align*}
with $$\lambda(\rho,\omega)=J(\rho,\omega)\rho^{n-1}.$$
For more details on  Cartan-Hadamard manifolds, we refer to  \cite{GHL,BC,H}.
\end{itemize}

\subsection{Basics on homogeneous Lie groups} 

A Lie group $\mathbb{G}$ (identified with $(\mathbb{R}^N, \circ)$) is called a homogeneous Lie group if it is equipped  with a dilation mapping $$ D_\lambda:\mathbb{R}^N \rightarrow \mathbb{R}^N,\quad\lambda >0,$$ defined as 
\begin{equation}
 D_\lambda(x)= (\lambda^{v_1}x_1,\lambda^{v_2}x_2,\ldots,\lambda^{v_N}x_N) ,\quad v_1,v_2,\dots,v_N > 0,
\end{equation} 
which is an automorphism of the group $\mathbb{G}$ for each $\lambda >0 $. At times, we will denote the image of $x\in \mathbb{G}$ under $D_\lambda$ by $\lambda(x)$ or, simply $\lambda x$.
The homogeneous dimension $Q$ of a homogeneous Lie group $\mathbb{G}$ is defined by 
$$Q=v_1+v_2+\dots+v_N.$$
It is well known that a homogeneous Lie group  is necessarily nilpotent and unimodular. The Haar measure $dx$ on $\mathbb{G}$ is nothing but the Lebesgue measure $\mathbb{R}^N$. 

Let us denote  the volume of a measurable set $\omega \subset \mathbb{G}$ by $|\omega|$. Then we have the following consequences: for $\lambda >0$
\begin{equation}
    |D_\lambda(\omega)|=\lambda^Q |\omega|  \quad \text{and} \quad \int_{\mathbb{G}} f(\lambda x) dx = \lambda^{-Q}\int_{\mathbb{G}} f(x) dx.
\end{equation}
A quasi-norm on $\mathbb{G}$ is any continuous non-negative function $ |\cdot|:\mathbb{G} \rightarrow [0,\infty)$ satisfying the following conditions:
\begin{itemize}
    \item[(i)] $|x|=|x^{-1}|$ for all $x \in \mathbb{G},$
    \item[(ii)] $|\lambda x|=\lambda |x|$\,\, for all \,\, $x \in \mathbb{G}$ and $\lambda >0,$
    \item[(iii)] $|x|=0 \iff x=0.$
\end{itemize}

If $\mathfrak{S}= \{x \in \mathbb{G}: |x|=1\} \subset \mathbb{G}$ is the unit sphere with respect to the quasi-norm, then there is a unique Radon measure $\sigma$ on $\mathfrak{S}$ such that for all $f \in L^1(\mathbb{G})$, we have the following polar decomposition
\begin{equation} \label{polar}
    \int_{\mathbb{G}} f(x) dx =\int_0^\infty \int_\mathfrak{S} f(ry) r^{Q-1}d\sigma(y) dr.
\end{equation}

\section{Main Results}

\subsection{Hardy inequality in metric measure space for the case $\mathbf{p=1}$ and $\mathbf{1\le q<\infty}$}
We will denote the closed balls by  $\mathbb{B}(a,r)$ in the metric measure space $\mathbb{X},$  where $a$ and $r$ are the centre and the radius of the ball, respectively. If $d$ is the metric on $\mathbb{X}$ then we can express $\mathbb{B}(a,r)$ by
\begin{align*}
  \mathbb{B}(a,r):=\{x\in \mathbb{X};d(x,a)\le r\},  
\end{align*}
and we write 
$|x|_a:=d(x,a)$ for the fixed point  $a\in \mathbb{X}.$
\newline Our first result is the characterization of weights $u$ and $v$ for the $L^1-$ integral Hardy type inequality on  metric measure spaces. We will also discuss the conjugate $L^1-$ integral Hardy on metric measure spaces which is  given in  Theorem \ref{conj}.

\begin{theorem}\label{Th1}
Let $1\leq q<\infty$ and let $\mathbb{X}$ be a metric measure space with a polar decomposition satisfying \eqref{polarform} at $a$. Let $u,v>0$ be measurable functions positive a.e. in $\mathbb{X}$ such that $u\in L^1(\mathbb{X}\backslash \{a\})$ 
and $v\in L^1_{loc}(\mathbb{X}).$ Then the  inequality
\begin{align}\label{eq1}
\left[\int_\mathbb{X}\left(\int_{\mathbb{B}(a,|x|_a)}|f(z)| dz\right)^q u(x)\,dx\right]^\frac{1}{q}\le C\left\{\int_\mathbb{X}|f(x)|\,v(x)\,dx\right\}
\end{align}
holds for all measurable functions $f:\mathbb{X}\rightarrow\mathbb{C}$ if and only if the following condition holds:
\begin{align}\label{eq2}
    A:=\sup_{R>0}\left(\int_{\mathbb{X}\backslash\mathbb{B}(a,R)}u(y)\,dy\right)^\frac{1}{q}\left(\sup_{y\in\mathbb{B}(a,R)} v^{-1}(y)\right)<\infty.
\end{align}
The best constant $C$ in \eqref{eq1} is given by $C=A.$
\end{theorem}

\begin{proof}
 For the proof of  Theorem \ref{Th1}, first we will prove that \eqref{eq2} implies \eqref{eq1}. We have

\begin{align}\label{eq3}
\left[\int_\mathbb{X}\left(\int_{\mathbb{B}(a,|x|_a)}|f(z)| dz\right)^q u(x)\,dx\right]^\frac{1}{q}= \left[\int_\mathbb{X}\left(\int_{\mathbb{X}}|f(z)|\mathbb{\chi}_{|z|_a\le{|x|_a}}(z,x)\,dz\right)^q u(x)\,dx\right]^\frac{1}{q},  
\end{align}
and by using Minkowski integral inequality in \eqref{eq3}, we get
\begin{align}\label{eq3.4}
  \nonumber&\left[\int_\mathbb{X}\left(\int_{\mathbb{X}}|f(z)|\mathbb{\chi}_{|z|_a\le{|x|_a}}(z,x)\,dz\right)^q u(x)\,dx\right]^\frac{1}{q} \le \int_\mathbb{X}\left(\int_{\mathbb{X}}|f(z)|^q\mathbb{\chi}_{|z|_a\le{|x|_a}}(z,x)\,u(x)dx\right)^\frac{1}{q}dz\\&\nonumber\quad\quad\quad\quad\quad\quad\quad\quad\quad\quad\quad=\int_\mathbb{X}|f(z)|\left(\int_{\mathbb{X}}\mathbb{\chi}_{|z|_a\le{|x|_a}}(z,x)\,u(x)dx\right)^\frac{1}{q}dz\\&\nonumber\quad\quad\quad\quad\quad\quad\quad\quad\quad\quad\quad\le \int_\mathbb{X}|f(z)|\left(\int_{\mathbb{X}\backslash\mathbb{B}(a,|z|_a)} u(x)dx\right)^\frac{1}{q}dz
  \\&\quad\quad\quad\quad\quad\quad\quad\quad\quad\quad\quad= \int_\mathbb{X}|f(z)|\,v(z)\left(\int_{\mathbb{X}\backslash\mathbb{B}(a,|z|_a)} u(x)dx\right)^\frac{1}{q}v^{-1}(z)dz.
\end{align}
  By simple calculation in \eqref{eq3.4} and applying \eqref{eq2}, we get
 \begin{align}\label{eq3.5}
  \left(\int_{\mathbb{X}\backslash\mathbb{B}(a,|z|_a)} u(x)dx\right)^\frac{1}{q}v^{-1}(z)\le\left(\int_{\mathbb{X}\backslash\mathbb{B}(a,|z|_a)} u(x)dx\right)^\frac{1}{q}\left(\sup_{y\in\mathbb{B}(a,|z|_a)}v^{-1}(y)\right) \le A . 
 \end{align}
 Again using \eqref{eq3.5} in \eqref{eq3.4}, we obtain
 \begin{align}
 \left[\int_\mathbb{X}\left(\int_{\mathbb{B}(a,|x|_a)}|f(z)| dz\right)^q u(x)\,dx\right]^\frac{1}{q}  \le A\int_\mathbb{X}|f(z)|\,v(z)\,dz. 
  \end{align}
  This  implies that \eqref{eq1} holds with $C=A.$ Then, the best constant $C$ in \eqref{eq1} satisfies $C \leq A.$

Conversely, denote
$S(R):=\sup_{z\in \mathbb{B}(a,R)}v^{-1}(z)$ for $R>0$ and then, for each $n\in\mathbb{N},$   define the following set $$\widetilde{P}_n:=\bigg\{z\in \mathbb{B}(a,R) ; v^{-1}(z)>S(R)-\frac{1}{n}\bigg\}.$$
  We note that $|\widetilde{P}_n|>0$ by the definition of $S(R)$   and   $\widetilde{P_n} \subset \mathbb{B}(a, R)$ implies that $|\widetilde{P}_n|<\infty$ for every $n \in \mathbb{N}.$
Now, for $R>0,$  we have
\begin{align}\label{eq3.7}
   \nonumber\bigg[\int_{\mathbb{X}\backslash \mathbb{B}(a,R)}|\widetilde{P}_n|^q\,u(x)\, dx\bigg]^\frac{1}{q}&= \bigg[\int_{\mathbb{X}\backslash \mathbb{B}(a,R)}\bigg(\int_{\mathbb{B}(a, R)}\chi_{\widetilde{P}_n}(y)dy\bigg)^q\,u(x)\, dx\bigg]^\frac{1}{q} \\&\le \nonumber \bigg[\int_{\mathbb{X}\backslash \mathbb{B}(a,R)}\bigg(\int_{\mathbb{B}(a, |x|_a)}\chi_{\widetilde{P}_n}(y)dy\bigg)^q\,u(x)\, dx\bigg]^\frac{1}{q}\\&\le\bigg[\int_{\mathbb{X}}\bigg(\int_{\mathbb{B}(a,|x|_a)}\chi_{\widetilde{P}_n}(y)dy\bigg)^q\,u(x)\, dx\bigg]^\frac{1}{q},
\end{align}
where in the second inequality we used that $|x|_a>R.$

 Using \eqref{eq1} in \eqref{eq3.7}, we obtain
 \begin{align}\label{eq3.8}
  \bigg[\int_{\mathbb{X}\backslash \mathbb{B}(a,R)}|\widetilde{P}_n|^q\,u(x)\, dx\bigg]^\frac{1}{q}\le C \int_{\mathbb{X}} \chi_{\widetilde{P}_n}(x)\, v(x)\,dx = C \int_{\widetilde{P}_n}  v(x) dx, 
 \end{align}
and we have 
\begin{align}\label{eq3.9}
 \int_{\widetilde{P}_n} v(x)\,dx   \le \bigg(S(R)-\frac{1}{n}\bigg)^{-1}\,|\widetilde{P}_n|.
\end{align}
 Again using \eqref{eq3.9} in \eqref{eq3.8}, we get
 \begin{align}
 \bigg[\int_{\mathbb{X}\backslash \mathbb{B}(a,R)}|\widetilde{P}_n|^q\,u(x)\, dx\bigg]^\frac{1}{q} \le C  \bigg(S(R)-\frac{1}{n}\bigg)^{-1}\,|\widetilde{P}_n|. 
 \end{align}
 As $n\rightarrow \infty,$
 we have
 \begin{align*}
  \bigg[\int_{\mathbb{X}\backslash \mathbb{B}(a,R)}u(x)\, dx\bigg]^\frac{1}{q} \le C  \bigg(S(R)\bigg)^{-1}  
 \end{align*}
 which says that
 \begin{align}
     \bigg[\int_{\mathbb{X}\backslash \mathbb{B}(a,R)}u(x)\, dx\bigg]^\frac{1}{q} S(R)\le C 
 \end{align} for all $R>0,$ which is our required result \eqref{eq2}. Indeed, this shows that 
 $$A:=\sup_{R>0}\left(\int_{\mathbb{X}\backslash\mathbb{B}(a,R)}u(y)\,dy\right)^\frac{1}{q}\left(\sup_{y\in\mathbb{B}(a,R)} v^{-1}(y)\right) \leq C$$
 as $C$ is independent of $R.$ Therefore, the best constant $C$ in \eqref{eq1} is given by $C=A.$ This completes the proof of this theorem. 
\end{proof}
Let us also briefly discuss the conjugate of the $L^1$- Hardy inequality in  Theorem \ref{Th1}:
\begin{theorem}\label{conj}
Let $1\leq q<\infty$ and let $\mathbb{X}$ be a metric measure space with a polar decomposition satisfying \eqref{polarform} at $a$. Let $u,v>0$ be measurable functions positive a.e. in $\mathbb{X}$ such that $u\in L_{loc}^1(\mathbb{X})$ 
and $v\in L^1_{loc}(\mathbb{X}\backslash \{a\}).$ Then the  inequality 
\begin{align}\label{eqconj}
\left[\int_\mathbb{X}\left(\int_{\mathbb{X}\backslash\mathbb{B}(a,|x|_a)}|f(z)| dz\right)^q u(x)\,dx\right]^\frac{1}{q}\le C\left\{\int_\mathbb{X}|f(x)|\,v(x)\,dx\right\},  
\end{align}
holds for all measurable functions $f:\mathbb{X}\rightarrow\mathbb{C}$ if and only if the following condition holds: 
\begin{align}
    A:=\sup_{R>0}\left(\int_{\mathbb{B}(a, R)}u(y)\,dy\right)^\frac{1}{q}\left(\sup_{y\in\mathbb{X}\backslash \mathbb{B}(a, R)}v^{-1}(y)\right)<\infty.
\end{align}
The best constant $C$ in \eqref{eqconj} is given by $C=A.$
\end{theorem}

\begin{proof}
The proof of the theorem  is similar  to the proof of  Theorem \ref{Th1}, and is omitted.\end{proof}

\subsection{Applications and examples}
In this section, we will give examples of the application of Theorem \ref{Th1} in the settings of homogeneous Lie groups, hyperbolic spaces and Cartan-Hadamard manifolds.
\begin{itemize}
    \item[(a)]\textbf{Homogeneous Lie groups :} Let $\mathbb{G}$ be a homogeneous Lie group of homogeneous dimension $Q,$ equipped with a quasi norm $|\cdot|$. We refer to \cite{RS,FR} for  more information on the set-up of homogeneous groups.
Particular examples of homogeneous Lie groups are the Euclidean space $\mathbb{R}^n$ (in which case $Q = n$), the
Heisenberg group, as well as general stratified Lie groups (homogeneous Carnot groups) and graded
groups. Here we use   notation for the quasi norm $|\cdot|,$  and denote $|x|_a$
 by $|x|$ with $a=0.$\\ Let us suppose that the  power weights are given by
$$u(x)=|x|^{-\beta q}\quad\text{and}\quad v(x)=|x|^{\alpha}.$$   For  the case $1\le q<\infty$ the inequality \eqref{eq1} holds  if and only if the following condition 
\begin{align}\label{eq4.1}
 A=\sup_{R>0}\bigg(|\mathfrak{S}|\int_R^\infty\rho^{-\beta q}\rho^{Q-1}\,d\rho \bigg)^\frac{1}{q} \bigg(\sup_{s \in (0, R)} s^{-\alpha}\bigg) <\infty,
\end{align}
holds,
where  $\mathfrak{S}$ is the area of the unit sphere in homogeneous Lie group $\mathbb{G}$ equipped with quasi norm $|\cdot|.$ For this supremum to be well defined, we need to have $\beta q>Q$ and $\alpha\le0.$ Again from \eqref{eq4.1}, we have
\begin{align}
   A=|\mathfrak{S}|^\frac{1}{q}\sup_{R>0}\bigg(\int_R^\infty\rho^{-\beta q+Q-1} d\rho\bigg)^\frac{1}{q}\bigg(R^{-\alpha}\bigg)
   =|\mathfrak{S}|^\frac{1}{q}\sup_{R>0}\bigg(\frac{R^{\frac{-\beta q+Q}{q}}}{|-\beta q+Q|^\frac{1}{q}}\bigg)\bigg(R^{-\alpha}\bigg),
\end{align}
which is finite if and only if $\alpha+\beta=\frac{Q}{q}.$\\
After the above conclusion, we obtain the following result:
\begin{cor}\label{cor4.1}
Let $\mathbb{G}$ be a homogeneous Lie group of homogeneous dimension $Q,$ equipped with a quasi norm $|\cdot|.$ Let $1\le q<\infty$ and let $\alpha, \beta \in \mathbb{R}.$ Then the inequality 
\begin{align}\label{eq4.3}
    \bigg(\int_\mathbb{G}\bigg(\int_{\mathbb{B}(0,|x|)} f(y)\, dy\bigg)^q|x|^{-\beta q} dx\bigg)^\frac{1}{q}\le C
   \bigg \{\int_\mathbb{G}|f(x)|\,|x|^{\alpha}dx\bigg\}
\end{align}
holds for all measurable functions $f:\mathbb{G}\rightarrow\mathbb{C}$ if and only if $\beta q>Q,\,\alpha<0 $ and $\alpha+\beta =\frac{Q}{q}.$ Moreover, the best constant  $C$ is  given by
\begin{align}
    C=|\mathfrak{S}|^\frac{1}{q}\left(\frac{1}{|\beta q-Q|^\frac{1}{q}}\right),
\end{align}
where $|\mathfrak{S}|$ is the area of the unit sphere in the homogeneous Lie group $\mathbb{G}$ with respect to the quasi norm $|\cdot|.$
\end{cor}
\item[(b)]
\textbf{Hyperbolic space:}
Let $\mathbb{H}^n$ be the hyperbolic space of dimension $n$ and let $a\in \mathbb{H}^n.$ Let us  take the weights  $$u(y)=\bigg(\sinh{|y|_a}\bigg)^{-\beta q}\quad\text{and}\quad v(y)=\bigg(\sinh{|y|_a}\bigg)^{\alpha}.$$
Using the polar decomposition in the first term of \eqref{eq2}, we get
\begin{align}\label{eq4.0}
 A\simeq \sup_{R>0} \bigg(\int_R^\infty\left(\sinh{\rho}\right)^{-\beta q+n-1}d\rho\bigg)^\frac{1}{q} \bigg(\sup_{y\in\mathbb{B}(0,R)}\left(\sinh{|y|_a}\right)^{-\alpha}\bigg).
\end{align} 

Next, we want to show that the supremum defined in \eqref{eq4.0} is finite.
For $R\gg1$, we have $ \sinh{z}\approx \exp{z}.$ Then  \eqref{eq4.0} can be written as 
\begin{align}\label{eq4.6}
\nonumber A&\simeq \begin{cases}
\sup_{R\gg1} \bigg(\int_{R}^\infty (\exp{\rho})^{-\beta q+n-1} d\rho\bigg)^\frac{1}{q}  &\quad\text{if}\quad \alpha \geq 0\\
\sup_{R\gg1} \bigg(\int_{R}^\infty (\exp{\rho})^{-\beta q+n-1} d\rho\bigg)^\frac{1}{q} \left(\exp{R}\right)^{-\alpha} &\quad\text{if}\quad \alpha \leq  0
\end{cases}
\\&\simeq \begin{cases}      \sup_{R\gg1}  \bigg(\exp{R\bigg)^{{\frac{-\beta q+n-1}{q}}}} &\quad\text{if}\quad \alpha \geq 0\\
 \sup_{R\gg1}  \bigg(\exp{R\bigg)^{{\frac{-\beta q+n-1}{q}}-\alpha}} &\quad\text{if}\quad \alpha \leq 0,
\end{cases}
\end{align} 
which are finite if and only if $\beta \ge \frac{n-1}{q}$ and  $\alpha+\beta \ge \frac{n-1}{q},$ respectively. 
\newline Similarly, if $R\ll1$ and for some small $L$ such as $\sinh{\rho}_{{R}\le\rho<L}\approx \rho,$
 \eqref{eq4.0} can be written as
\begin{align}\label{eq4.7}
  \nonumber &A\simeq \sup_{R\ll1}\bigg(\int_{R}^\infty  (\sinh{\rho})^{-\beta q+n-1}d\rho\bigg)^\frac{1}{q}\bigg(\sup_{y\in \mathbb{B}(0, R)}(\sinh{|y|_a})^{-\alpha}\bigg)\\&\nonumber\simeq \sup_{R \ll 1}\bigg(\int_{R}^L(\sinh{\rho})^{-\beta q+n-1}d\rho+\int_L^\infty(\sinh{\rho})^{-\beta q+n-1}d\rho\bigg)^\frac{1}{q}\bigg(\sup_{y\in \mathbb{B}(0,R)}(\sinh{|y|_a})^{-\alpha}\bigg)\\&\
  \nonumber\simeq \sup_{R\ll1}\bigg(\int_{R}^L({\rho})^{-\beta q+n-1}d\rho+\int_L^\infty(\exp{\rho})^{-\beta q+n-1}d\rho\bigg)^\frac{1}{q}\bigg({R}^{-\alpha}\bigg)\\&\nonumber\approx
  \sup_{R\ll1}\bigg(\frac{1}{-\beta q+n}\rho^{-\beta q+n}\bigg|_{R}^L+\frac{1}{(-\beta q+n-1)}\left(\exp{\rho}\right)^{-\beta q+n-1}\big|_L^\infty\bigg)^\frac{1}{q}\bigg({R}^{-\alpha}\bigg)\\&\approx \sup_{R\ll1}\bigg({R}^{-\beta q+n}+C_L\bigg)^\frac{1}{q}\left(R^{-\alpha}\right),
\end{align} 
where $C_L$ is the  remainder term and the condition $\alpha \leq 0$ is used in third inequality.
Now if $-\beta q+n\ge0,$ then again from \eqref{eq4.7}
this is $\approx  \sup_{R\ll1}R^{-\alpha}, $  which is finite if and only if $\alpha\le0.$ At the same time if $-\beta q+n<0,$ then we observe that \eqref{eq4.7} is $ \approx\sup_{R\ll1}R^{\frac{-\beta q+n}{q}-\alpha},$ which is finite if and only if $\frac{-\beta q+n}{q}-\alpha\ge 0$ which is further equivalent to $\alpha+\beta \leq \frac{n}{q}.$ 
\newline Therefore, after the above conclusion we obtain the following result.

\begin{cor}\label{cor4.2}
Let $\mathbb{H}^n$ be the hyperbolic space, $a\in \mathbb{H}^n.$  Let $|x|_a$ denote the hyperbolic distance from $x$ to $a.$ Let  $\alpha,\beta\in \mathbb{R}$ and $1\le q<\infty.$ Then the following inequality 
\begin{align}
    \bigg(\int_{\mathbb{H}^n}\bigg(\int_{\mathbb{B}(0,|x|_a)}|f(y)|dy\bigg)^q\left(\sinh{|x|_a}\right)^{-\beta q}dx\bigg)^\frac{1}{q}\le C\bigg\{\int_{\mathbb{H}^n}|f(x)|\left(\sinh{|x|_a}\right)^\alpha dx\bigg\}
\end{align}
holds for all measurable functions $f:\mathbb{H}^n\rightarrow\mathbb{C}$ if and only if the parameters $\alpha,\beta$ satisfy either of the following conditions:
\begin{itemize}
    \item[(a)] for $-\beta q+n\ge 0,$ if and only if $ \alpha \leq 0$ and $\alpha+\beta\ge \frac{n-1}{q},$
\item[(b)] for $-\beta q+n< 0,$ if and only if $\alpha \le 0$ and   $\frac{n-1}{q} \leq \alpha+\beta\leq \frac{n}{q}.$
\end{itemize}
\end{cor}
\item[(c)]
\textbf{Cartan-Hadamard Manifold:} 
 Let $(M,g)$ be a Cartan-Hadamard manifold with constant sectional curvature $\mathbb{K}_M.$ Let $J(t,\omega)$ is a function of $t$ only. More precisely, if $\mathbb{K}_M=-b$ for $b\ge 0,$ then $J(t,\omega)=1$ if $b=0,$ and $J(t,\omega)=\left(\frac{\sinh{\sqrt{b}t}}{\sqrt{b}t}\right)^{n-1}$ for $b>0.$
\newline When $b=0,$ then let us assume that $u(x)=|x|_a^{-\beta q}$ and $v(x)=|x|_a^\alpha$ then the inequality \eqref{eq2} holds for $1\le q<\infty$ if and only if 
\begin{align}\label{eq4.5}
    A=\sup_{R>0}\bigg(\int_{M\backslash\mathbb{B}(0,R)}|z|_a^{-\beta q} dz\bigg)^\frac{1}{q}\bigg(\sup_{z\in\mathbb{B}(a, R)}|z|_a^{-\alpha}\bigg)<\infty.
\end{align}
After changing to polar coordinates in \eqref{eq4.5} first term, we get 
\begin{align}
  \sup_{R>0}\bigg(\int_{R}^\infty\rho^{-\beta q+n-1} d\rho\bigg)^\frac{1}{q}\bigg(\sup_{z\in\mathbb{B}(a,R)}|z|_a^{-\alpha}\bigg),  \end{align} 
 which is finite if and only if conditions of Corollary \ref{cor4.1} hold with $Q=n.$ \\
When $b>0$ and if we take weights $u(x)=(\sinh{\sqrt{b}|x|_a})^{-\beta q}$ and $v(x)=(\sinh{\sqrt{b}|x|_a})^\alpha,$ then the inequality \eqref{eq2} holds for $1\le q<\infty$ if and only if
\begin{align*}
 \sup_{R>0}\bigg(\int_{M\backslash\mathbb{B}(a,R)} (\sinh{\sqrt{b}|z|_a})^{-\beta q}dz\bigg)^\frac{1}{q}\bigg(\sup_{z\in \mathbb{B}(0, R)}(\sinh{\sqrt{b}|z|_a})^{-\alpha}\bigg)<\infty.  
\end{align*}
Again changing to polar coordinates,
\begin{align*}
 &\sup_{R>0}\bigg(\int_{R}^\infty (\sinh{\sqrt{b}\mu})^{-\beta q}\bigg(\frac{\sinh{b}\mu}{\sqrt{b}\mu}\bigg)^{n-1} \mu^{n-1}\,d\mu\bigg)^\frac{1}{q}\bigg(\sup_{z\in \mathbb{B}(0, R)}(\sinh{\sqrt{b}|z|_a})^{-\alpha}\bigg)\\& \simeq
 \sup_{R>0}\bigg(\int_{R}^\infty(\sinh{\sqrt{b}\mu})^ {-\beta q+n-1} d\mu\bigg)^\frac{1}{q} \bigg(\sup_{z\in \mathbb{B}(0, R)} (\sinh{\sqrt{b}|z|_a})^{-\alpha}\bigg),
\end{align*}
 \noindent which has the same condition as in the case of the hyperbolic space as in Corollary \ref{cor4.2}.
\end{itemize}

\section*{acknowledgement} 
This paper was completed when the second author was visiting Ghent University, Belgium. She is very grateful to Ghent Analysis \& PDE centre, Ghent University, Belgium for the financial support and warm hospitality during her research visit.
MR is supported by the FWO Odysseus 1 grant G.0H94.18N: Analysis and Partial
Differential Equations, the Methusalem programme of the Ghent University Special Research
Fund (BOF) (Grant number 01M01021) and by EPSRC grant
EP/R003025/2. AS is supported by UGC Non-net fellowship from Banaras Hindu University, India.

\end{document}